\newcommand{\relmiddle}[1]{\mathrel{}\middle#1\mathrel{}}
\newcommand{\scalarprod}[2]{#1\cdot#2}

\documentclass{ws-ijnt}

\begin{document}

\markboth{Ken Yamamoto}{On exceptionality of dimension three in terms of lattice angles}

%
\catchline{}{}{}{}{}
%

\title{On exceptionality of dimension three in terms of lattice angles}

\author{Ken Yamamoto\footnote{Senbaru 1, Nishihara, Okinawa 903--0213, Japan}}

\address{Faculty of Science, University of the Ryukyus, Senbaru 1, Nishihara, Okinawa 903--0213, Japan\\
\email{yamamot@sci.u-ryukyu.ac.jp} }

\maketitle


\begin{abstract}
We investigate the lattice angle, namely angles between two vectors with integer coordinates.
We focus on the set of angles between a fixed integer vector and other integer vectors.
For non-three-dimensional lattices, we proved that this set contains all lattice angles, irrespective of the fixed vector choice.
In contrast, for the three-dimensional lattice, we proved that this set of angles cannot cover all lattice angles, for any fixed vector.
Thus, only the three-dimensional lattice is an exception.
We further provide the condition for a given three-dimensional integer vector to intersect another integer vector at a given angle, which involves a number-theoretic property of the squared norm of the given vector and the squared tangent of the given angle.
\end{abstract}

\keywords{Lattice; Lattice angle; Quadratic form; Hilbert symbol.}

\ccode{Mathematics Subject Classification 2010: 11H06}

\section{Introduction}
The integer lattice $\mathbb{Z}^n$ is a simple and basic mathematical structure in which geometry, number theory, algebra, combinatorics, and other mathematical branches interact~\cite{Erdos, Olds}.
For example, Eisenstein's proof of quadratic reciprocity was performed by counting lattice points in a triangular region~\cite{Lemmermeyer}.
Minkowski initiated the ``geometry of numbers,'' and his theorem on convex sets has been applied in the proof of several theorems in number theory~\cite{Matousek}.
Later, Siegel and Mordell provided in-depth results for lattice or rational points on elliptic curves~\cite{Silverman}.
Currently, mathematics of lattices, including ones other than $\mathbb{Z}^n$, has attracted interest in the fields of applied mathematics, engineering, and the natural sciences,
such as cryptography~\cite{Micciancio}, computer graphics~\cite{Salomon}, and materials science~\cite{MaciaBarber}.

The mathematics of lattice polygons and polyhedra has been developed in many aspects.
Here, a lattice polygon and polyhedron are defined as a polygon and polyhedron whose vertices are all lattice points, respectively.
One of the most famous results is Pick's theorem~\cite{Beck}, which computes the area of lattice polygons in $\mathbb{R}^2$ using the number of lattice points in the interior and that on the boundary.
This theorem is used to prove Minkowski's theorem using the Farey sequence~\cite{Hardy}, and is sometimes used as teaching material in mathematics education~\cite{Johnson}.
Various extensions of Pick's theorem have been investigated, such as hexagonal~\cite{Ding1988}, triangular~\cite{Polis}, and other lattices~\cite{Ding1987}, as well as the volume of a polyhedron~\cite{Macdonald, Reeve}.
The Ehrhart polynomial~\cite{Beck} is an extension of Pick's theorem for higher-dimensional lattice objects, and it appears in the study of toric varieties~\cite{Fulton}.

Because of the discrete structure of the lattice, realizable lattice polygons and polyhedra are limited.
For lattice regular polygons in $\mathbb{Z}^n$, it is well known that only the square exists for $n=2$~\cite{Scherrer, OLoughlin}, and regular triangle, square, and regular hexagon are allowed for $n\ge3$~\cite{Schoenberg}.
As a corollary, a lattice regular polyhedron in $\mathbb{R}^3$ was proven to be either a regular tetrahedron, cube, or regular octahedron.
It was proven that if the distance between two three-dimensional lattice points is an integer, there exists a lattice cube containing these two points as vertices~\cite{Parris}.
Furthermore, the condition that there exists a regular $n$-simplex whose vertices are lattice points in $\mathbb{R}^n$ depends on the number-theoretic properties of $n$~\cite{Schoenberg}.
For example, if $n$ is even, a lattice regular $n$-simplex exists if and only if $n+1$ is a square number.
Thus, properties of lattice objects may involve the lattice dimension.

The lattice angle, formed by a fixed lattice point (the origin) and two lattice points, is a fundamental quantity in lattice geometry, and this is the main object of this study (see Fig.~\ref{fig1} for example).
To define the lattice angles in terms of vector algebra, we regard lattice $\mathbb{Z}^n$ as a subset of the Euclidean space $\mathbb{R}^n$.
We identify each lattice point $(a_1,\ldots, a_n)\in\mathbb{Z}^n$ with a vector from the origin $(0,\ldots, 0)$ to this point.
The vector corresponding to a lattice point is referred to as the ``integer vector'' in this study.
We use bold symbols, such as $\boldsymbol{a}^{(n)}$, for $n$-dimensional vectors and write as $\boldsymbol{a}^{(n)}=(a_1,\ldots, a_n)$.
In addition, we introduce the notation for the angle between two vectors $\boldsymbol{a}^{(n)}$ and $\boldsymbol{b}^{(n)}$ [$\boldsymbol{a}^{(n)}, \boldsymbol{b}^{(n)}\ne\boldsymbol{0}^{(n)}:=(0,\ldots,0)$] as follows:
\[
\angle(\boldsymbol{a}^{(n)}, \boldsymbol{b}^{(n)}):=\arccos\frac{\scalarprod{\boldsymbol{a}^{(n)}}{\boldsymbol{b}^{(n)}}}{\|\boldsymbol{a}^{(n)}\| \|\boldsymbol{b}^{(n)}\|},
\]
where
\[
\|\boldsymbol{a}^{(n)}\|=\left(\sum_{i=1}^n a_i^2\right)^{1/2},\quad
\scalarprod{\boldsymbol{a}^{(n)}}{\boldsymbol{b}^{(n)}}=\sum_{i=1}^n a_ib_i
\]
are the Euclidean norm and scalar product, respectively.

\begin{figure}[tb!]\centering
\includegraphics[bb=0 0 104 100, clip]{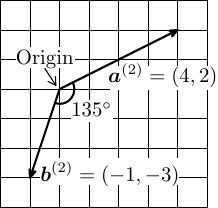}
\caption{
Example of the lattice angle in $\mathbb{Z}^2$, formed by two integer vectors $\boldsymbol{a}^{(2)}=(4,2)$ and $\boldsymbol{b}^{(2)}=(-1,-3)$.
The angle in this case is $\angle(\boldsymbol{a}^{(2)}, \boldsymbol{b}^{(2)})=3\pi/4=135^\circ$.
}
\label{fig1}
\end{figure}

Let $\Theta_n$ denote the set of the lattice angles in $\mathbb{Z}^n$.
Clearly, $\Theta_n\subseteq\Theta_m$ for any $n\le m$.
More specifically, Beeson~\cite{Beeson} determined $\Theta_n$ for each dimension $n$, as follows:
\begin{enumerate}
\item $\Theta_2=\left\{\dfrac{\pi}{2}\right\}\cup\left\{\theta\mid \tan\theta\in\mathbb{Q}\right\}$.
\item $\Theta_3=\Theta_4=\left\{\dfrac{\pi}{2}\right\}\cup\left\{\theta\relmiddle| \tan^2\theta=\dfrac{r^2+s^2+t^2}{u^2}, r,s,t,u\in\mathbb{Z}\right\}$.
\item $\Theta_5=\Theta_6=\cdots=\left\{\dfrac{\pi}{2}\right\}\cup\left\{\theta\mid \tan^2\theta\in\mathbb{Q}\right\}$.
\end{enumerate}
Thus, the relation $\Theta_2\subsetneq\Theta_3=\Theta_4\subsetneq\Theta_5=\Theta_6=\cdots$ is obtained.
According to Legendre's three-square theorem~\cite{Serre}, $\theta\in\Theta_3\setminus\{\pi/2\}$ implies that $\tan^2\theta$ does not have the form $(8m+7)x^2$, where $m$ is an integer and $x$ is a positive rational number.

Beeson's theorem for $n=3$ shows $\pi/3\in\Theta_3$ because $\tan^2(\pi/3)=3=(1^2+1^2+1^2)/1$.
In fact, one can easily find that the two vectors $\boldsymbol{a}^{(3)}=(1,1,0)$ and $\boldsymbol{b}^{(3)}=(0,1,1)$ form $\angle(\boldsymbol{a}^{(3)}, \boldsymbol{b}^{(3)})=\pi/3$.
However, for some integer vectors, e.g., $\boldsymbol{e}_1^{(3)}=(1,0,0)$ shown in the following proposition, no integer vectors can be found to make angle $\pi/3$.
\begin{proposition}
For $\boldsymbol{e}_1^{(3)}=(1,0,0)$, there exists no integer vector $\boldsymbol{v}^{(3)}$ such that $\angle(\boldsymbol{e}_1^{(3)}, \boldsymbol{v}^{(3)})=\pi/3$.
\label{prop1.1}
\end{proposition}
\begin{proof}
Suppose an integer vector $\boldsymbol{v}^{(3)}=(v_1, v_2, v_3)$ satisfies $\angle(\boldsymbol{e}_1^{(3)}, \boldsymbol{v}^{(3)})=\pi/3$; then, $3v_1^2=v_2^2+v_3^2$ is obtained.
This equation states that the natural number $3v_1^2$ is represented as the sum of two squares, but this contradicts Fermat's two-squares theorem~\cite{Mollin}.
\end{proof}
As mentioned above, even if $\theta\in\Theta_3$, not all integer vectors can form the angle $\theta$ with another integer vector.
%
To formulate this observation more precisely, we introduce the following notation.
\begin{definition}
For an $n$-dimensional integer vector $\boldsymbol{a}^{(n)}\ne\boldsymbol{0}^{(n)}$,
\[
\Theta_n(\boldsymbol{a}^{(n)}):=\{\angle(\boldsymbol{a}^{(n)}, \boldsymbol{b}^{(n)})\mid \boldsymbol{b}^{(n)}\in\mathbb{Z}^n\setminus\{\boldsymbol{0}^{(n)}\}\}
\]
as the set of angles between $\boldsymbol{a}^{(n)}$ and other integer vectors.
\end{definition}
By definition,
\[
\Theta_n(\boldsymbol{a}^{(n)})\subseteq\Theta_n
\]
holds for any $\boldsymbol{a}^{(n)}\ne\boldsymbol{0}^{(n)}$ and
\[
\Theta_n = \bigcup_{\substack{\boldsymbol{a}^{(n)}\in\mathbb{Z}^n\\ \boldsymbol{a}^{(n)}\ne\boldsymbol{0}^{(n)}}} \Theta_n(\boldsymbol{a}^{(n)})
\]
for each dimension $n$.
Moreover, we have the following main theorem.
\begin{theorem}
\begin{enumerate}
\item\label{thm1.3-1} For $n\ne3$,
\[
\Theta_n(\boldsymbol{a}^{(n)})=\Theta_n
\]
for any $\boldsymbol{a}^{(n)}\in\mathbb{Z}^n\setminus\{\boldsymbol{0}^{(n)}\}$.
That is, for each integer vector $\boldsymbol{a}^{(n)}$, any angle $\theta\in\Theta_n$ is attained as $\theta=\angle(\boldsymbol{a}^{(n)}, \boldsymbol{b}^{(n)})$ by choosing an appropriate integer vector $\boldsymbol{b}^{(n)}$.
\item\label{thm1.3-2} For $n=3$,
\[
\Theta_3(\boldsymbol{a}^{(3)})\subsetneq\Theta_3
\]
for any $\boldsymbol{a}^{(3)}\in\mathbb{Z}^3\setminus\{\boldsymbol{0}^{(3)}\}$.
That is, for each integer vector $\boldsymbol{a}^{(3)}$, there exists an angle $\theta\in\Theta_3$, such that $\theta=\angle(\boldsymbol{a}^{(n)}, \boldsymbol{b}^{(n)})$ cannot be realized for any integer vector $\boldsymbol{b}^{(n)}$.
\end{enumerate}
\label{thm1}
\end{theorem}

From this theorem, the situation as in Proposition~\ref{prop1.1}, where $\pi/3\in\Theta_3$ but $\Theta_3(1,0,0)\not\ni\pi/3$, occurs only in the three-dimensional lattice.
That is, the property of $\Theta_n(\boldsymbol{a}^{(n)})$ is completely distinct only in dimension $n=3$.
The proof of this theorem is provided in the next section.
From Theorem~\ref{thm1}(\ref{thm1.3-2}), a next natural problem is to find the (necessary and sufficient) condition for $\boldsymbol{a}^{(3)}$ such that $\Theta_3(\boldsymbol{a}^{(3)})$ contains a given $\theta\in\Theta_3$.
Likewise, the condition for $\theta\in\Theta_3$, such that $\Theta_3(\boldsymbol{a}^{(3)})$ for a given $\boldsymbol{a}^{(3)}$ contains $\theta$ is also a natural problem.
In \S\ref{sec3}, we show that the condition for $\theta\in\Theta_3(\boldsymbol{a}^{(3)})$ is associated with number-theoretic properties of $\tan^2\theta$ and $\|\boldsymbol{a}^{(3)}\|^2$, including its prime factorization.

\section{Proof of Theorem~\ref{thm1}}\label{sec2}
In this section, we provide the proof of Theorem~\ref{thm1}.

First, we can easily show that $0, \pi\in\Theta_n(\boldsymbol{a}^{(n)})$ for any integer vector $\boldsymbol{a}^{(n)}\ne\boldsymbol{0}^{(n)}$, because $\angle(\boldsymbol{a}^{(n)}, \boldsymbol{a}^{(n)})=0$ and $\angle(\boldsymbol{a}^{(n)}, -\boldsymbol{a}^{(n)})=\pi$.
Furthermore, we can prove that $\pi/2\in\Theta_n(\boldsymbol{a}^{(n)})$ for any integer vector $\boldsymbol{a}^{(n)}=(a_1, a_2,\ldots, a_n)$ of any dimension $n$.
If $a_1=0$, the integer vector $(1,0,\ldots,0)$ is perpendicular to $\boldsymbol{a}^{(n)}$;
otherwise (if $a_1\ne0$), the integer vector $(-a_2, a_1,0,\ldots, 0)$ is perpendicular to $\boldsymbol{a}^{(n)}$. 

Thus, to prove Theorem~\ref{thm1}(\ref{thm1.3-1}), we demonstrate that $\Theta_n(\boldsymbol{a}^{(n)})$ contains any $\theta\in\Theta_n$ other than $0, \pi/2$, and $\pi$.
The proof can be separated into three cases: (1-1) $n=2$, (1-2) $n=4$, and (1-3) $n\ge5$.
We note that these cases correspond to Beeson's $\Theta_n$: $\Theta_2\subsetneq\Theta_4\subsetneq\Theta_5=\Theta_6=\cdots$.

(1-1) The case $n=2$ is the simplest.
Each angle $\theta\in\Theta_2\setminus\{\pi/2\}$ has a rational tangent, so that we can write $\tan\theta=v/u$ for $u,v\in\mathbb{Z}$ and $u\ne0$.
By putting
\begin{equation}
b_1=a_1u+a_2v,\quad b_2=a_2u-a_1v,
\label{eq2.1}
\end{equation}
it is straightforward to show that $\boldsymbol{b}^{(2)}=(b_1,b_2)$ is an integer vector such that $\angle(\boldsymbol{a}^{(2)}, \boldsymbol{b}^{(2)})=\theta$.
This argument is similar to the proof by Beeson~\cite{Beeson}.

(1-2) Next, we present the case $n=4$.
Let $I_4$ be the $4\times4$ identity matrix, and introduce three orthogonal matrices as follows:
\[
U_1=\begin{pmatrix}0 & -1 & 0 & 0\\
1 & 0 & 0 & 0\\
0 & 0 & 0 & -1\\
0 & 0 & 1 & 0
\end{pmatrix},
\quad
U_2=\begin{pmatrix}0 & 0 & -1 & 0\\
0 & 0 & 0 & 1\\
1 & 0 & 0 & 0\\
0 & -1 & 0 & 0
\end{pmatrix},
\quad
U_3=\begin{pmatrix}0 & 0 & 0 & -1\\
0 & 0 & -1 & 0\\
0 & 1 & 0 & 0\\
1 & 0 & 0 & 0
\end{pmatrix}.
\]
For any four-dimensional integer vector $\boldsymbol{a}^{(4)}=(a_1, a_2, a_3, a_4)\ne(0,0,0,0)$, we set
\begin{align*}
\boldsymbol{x}_0^{(4)}&=\boldsymbol{a}^{(4)}I_4=(a_1, a_2, a_3, a_4),\\
\boldsymbol{x}_1^{(4)}&=\boldsymbol{a}^{(4)}U_1=(a_2, -a_1, a_4, -a_3),\\
\boldsymbol{x}_2^{(4)}&=\boldsymbol{a}^{(4)}U_2=(a_3, -a_4, -a_1, a_2),\\
\boldsymbol{x}_3^{(4)}&=\boldsymbol{a}^{(4)}U_3=(a_4, a_3, -a_2, -a_1).
\end{align*}
These four vectors are integer vectors having equal norm $\|\boldsymbol{x}_k^{(4)}\|=\|\boldsymbol{a}^{(4)}\|$ for $k=0$, $1$, $2$, and $3$
and form an orthogonal basis for $\mathbb{R}^4$.
Using arbitrary integers $u,r,s,t\in\mathbb{Z}$, we construct an integer vector
\begin{equation}
\boldsymbol{b}^{(4)}=u\boldsymbol{x}_0^{(4)}+r\boldsymbol{x}_1^{(4)}+s\boldsymbol{x}_2^{(4)}+t\boldsymbol{x}_3^{(4)},
\label{eq:b}
\end{equation}
i.e., $u$, $r$, $s$, and $t$ are the coordinates of $\boldsymbol{b}^{(4)}$ in the orthogonal basis $\{\boldsymbol{x}_0^{(4)}, \boldsymbol{x}_1^{(4)}, \boldsymbol{x}_2^{(4)}, \boldsymbol{x}_3^{(4)}\}$.
The angle $\theta$ between $\boldsymbol{a}^{(4)}$ and $\boldsymbol{b}^{(4)}$ satisfies 
\[
\cos\theta=\frac{\boldsymbol{a}^{(4)}\cdot\boldsymbol{b}^{(4)}}{\|\boldsymbol{a}^{(4)}\| \|\boldsymbol{b}^{(4)}\|}=\frac{u}{\sqrt{u^2+r^2+s^2+t^2}},
\]
which implies
\[
\tan^2\theta=\frac{r^2+s^2+t^2}{u^2}.
\]
Thus, $\Theta_4(\boldsymbol{a}^{(4)})$ is proven to contain each $\theta\in\Theta_4$.

In other words, for a given integer vector $\boldsymbol{a}^{(4)}\ne(0,0,0,0)$, an orthogonal basis of integer vectors with equal norm including $\boldsymbol{a}^{(4)}$ can be constructed, and the vector $\boldsymbol{b}^{(4)}$ having the coordinate $(u,r,s,t)$ in this basis satisfies $\tan^2\angle(\boldsymbol{a}^{(4)}, \boldsymbol{b}^{(4)})=(r^2+s^2+t^2)/u^2$.

\begin{remark}
For $\boldsymbol{a}^{(4)}=(1,1,1,1)$, the vector $\boldsymbol{b}^{(4)}$ in Eq.~\eqref{eq:b} is written as
\begin{equation}
\boldsymbol{b}^{(4)}=
\begin{pmatrix}
u & r & s & t
\end{pmatrix}
\begin{pmatrix}
1 & 1 & 1 & 1\\
1 & -1 & 1 & -1\\
1 & -1 & -1 & 1\\
1 & 1 & -1 & -1
\end{pmatrix}.
\label{eq:Hadamard}
\end{equation}
The $4\times4$ matrix on the right-hand side is a Hadamard matrix of order $4$~\cite{Horadam};
a Hadamard matrix is a square matrix whose entries are either $1$ or $-1$ and the row vectors are mutually orthogonal.

For general $\boldsymbol{a}^{(4)}=(a_1, a_2, a_3, a_4)$, $\boldsymbol{b}^{(4)}$ becomes
\[
\boldsymbol{b}^{(4)}=
\begin{pmatrix}
u & r & s & t
\end{pmatrix}
\begin{pmatrix}
\boldsymbol{x}_0^{(4)}\\
\boldsymbol{x}_1^{(4)}\\
\boldsymbol{x}_2^{(4)}\\
\boldsymbol{x}_3^{(4)}
\end{pmatrix}
=
\begin{pmatrix}
u & r & s & t
\end{pmatrix}
\begin{pmatrix}
a_1 & a_2 & a_3 & a_4\\
a_2 & -a_1 & a_4 & -a_3\\
a_3 & -a_4 & -a_1 & a_2\\
a_4 & a_3 & -a_2 & -a_1
\end{pmatrix}.
\]
The following four vectors
\newlength{\minusskip}\settowidth{\minusskip}{$-$}
\begin{align*}
\boldsymbol{y}_1^{(4)}&=\frac{1}{4}(a_1+a_2+a_3+a_4, -a_1+a_2+a_3-a_4, -a_1-a_2+a_3+a_4, -a_1+a_2-a_3+a_4),\\
\boldsymbol{y}_2^{(4)}&=\frac{1}{4}(a_1-a_2-a_3+a_4, \hspace{\minusskip}a_1+a_2+a_3+a_4, \hspace{\minusskip}a_1-a_2+a_3-a_4, -a_1-a_2+a_3+a_4),\\
\boldsymbol{y}_3^{(4)}&=\frac{1}{4}(a_1+a_2-a_3-a_4, -a_1+a_2-a_3-a_4, \hspace{\minusskip}a_1+a_2+a_3+a_4, \hspace{\minusskip}a_1-a_2-a_3+a_4),\\
\boldsymbol{y}_4^{(4)}&=\frac{1}{4}(a_1-a_2+a_3-a_4, \hspace{\minusskip}a_1+a_2-a_3-a_4, -a_1+a_2+a_3-a_4, \hspace{\minusskip}a_1+a_2+a_3+a_4)
\end{align*}
form an orthogonal basis with equal norm ($\|\boldsymbol{y}_k^{(4)}\|=\|\boldsymbol{a}^{(4)}\|/2$).
The vector $\boldsymbol{b}^{(4)}$ can be expressed as
\[
\boldsymbol{b}^{(4)}=
\begin{pmatrix}
u & r & s & t
\end{pmatrix}
\begin{pmatrix}
1 & 1 & 1 & 1\\
1 & -1 & 1 & -1\\
1 & -1 & -1 & 1\\
1 & 1 & -1 & -1
\end{pmatrix}
\begin{pmatrix}
\boldsymbol{y}_1^{(4)}\\
\boldsymbol{y}_2^{(4)}\\
\boldsymbol{y}_3^{(4)}\\
\boldsymbol{y}_4^{(4)}
\end{pmatrix}.
\]
The same Hadamard matrix as in Eq.~\eqref{eq:Hadamard} appears for general $\boldsymbol{a}^{(4)}$ in the basis $\{\boldsymbol{y}^{(4)}_1, \boldsymbol{y}^{(4)}_2, \boldsymbol{y}^{(4)}_3, \boldsymbol{y}^{(4)}_4\}$.

In the two-dimensional case, the vector $\boldsymbol{b}^{(2)}=(b_1, b_2)$ given by Eq.~\eqref{eq2.1} has an expression using a Hadamard matrix of order $2$:
\[
\boldsymbol{b}^{(2)}=
\begin{pmatrix}
u & v
\end{pmatrix}
\begin{pmatrix}
a_1 & a_2\\
a_2 & -a_1
\end{pmatrix}
=
\begin{pmatrix}
u & v
\end{pmatrix}
\begin{pmatrix}
1 & 1\\
1 & -1
\end{pmatrix}
\begin{pmatrix}
\boldsymbol{y}_1^{(2)}\\
\boldsymbol{y}_2^{(2)}
\end{pmatrix},
\]
where
\[
\boldsymbol{y}_1^{(2)}=\frac{1}{2}(a_1+a_2, -a_1+a_2),\quad
\boldsymbol{y}_2^{(2)}=\frac{1}{2}(a_1-a_2, a_1+a_2)
\]
are orthogonal and have equal norm $\|\boldsymbol{y}_1^{(2)}\|=\|\boldsymbol{y}_2^{(2)}\|=\|\boldsymbol{a}^{(2)}\|/\sqrt{2}$.

Therefore, the construction of an integer vector $\boldsymbol{b}^{(n)}$ that intersects at a given angle with a given integer vector $\boldsymbol{a}^{(n)}$ for $n=2$ and $4$ involves Hadamard matrices.
In the above computation, the author constructed the orthogonal basis $\{\boldsymbol{y}_1^{(4)}, \boldsymbol{y}_2^{(4)}, \boldsymbol{y}_3^{(4)}, \boldsymbol{y}_4^{(4)}\}$ after much trial and error.
The author believes that further knowledge regarding integer vectors and Hadamard matrices provides a general and systematic method to generate such an orthogonal basis.
\end{remark}

(1-3) The proof of Theorem~\ref{thm1}(\ref{thm1.3-1}) is complete for case $n\ge5$.
It suffices to show the existence of vector $\boldsymbol{v}^{(n)}$ whose elements $v_1,\ldots,v_n$ are \textit{rational numbers} satisfying $\angle(\boldsymbol{a}^{(n)}, \boldsymbol{v}^{(n)})=\theta\in\Theta_n$.
Note that an integer vector in the same direction as $\boldsymbol{v}^{(n)}$ can be constructed by multiplying a common multiple of the denominators of $v_1,\ldots,v_n$.
This rational vector $\boldsymbol{v}^{(n)}$ satisfies
\[
(\boldsymbol{v}^{(n)}\cdot\boldsymbol{a}^{(n)})^2 - \|\boldsymbol{v}^{(n)}\|^2 \|\boldsymbol{a}^{(n)}\|^2\cos^2\theta = 0.
\]
The left-hand side is a quadratic form in $v_1,\ldots,v_n$, and the element $m_{ij}$ of the coefficient matrix $M$ is
\begin{equation}
m_{ij}=a_ia_j-\delta_{ij}\|\boldsymbol{a}^{(n)}\|^2\cos^2\theta,
\label{eq1}
\end{equation}
where $\delta_{ij}$ is the Kronecker delta.
Note that $\cos^2\theta=1/(\tan^2\theta+1)$ is rational because $\tan^2\theta$ for $\theta\in\Theta_n$ is rational.
Hence, the problem is reduced to determining whether the quadratic form $\sum_{i,j}m_{ij} v_i v_j$ over $\mathbb{Q}$ represents $0$ in $\mathbb{Q}$.

We apply the following theorem for quadratic forms:
\begin{theorem}[Meyer's theorem; see Serre~\cite{Serre} \S3.2]
A quadratic form of rank $\ge5$ represents $0$ in $\mathbb{Q}$ if and only if it is indefinite.
\end{theorem}

From Eq.~\eqref{eq1}, we find that $\boldsymbol{a}^{(n)}$ is an eigenvector of $M$ and the corresponding eigenvalue is $\|\boldsymbol{a}^{(n)}\|^2\sin^2\theta>0$.
Any vector perpendicular to $\boldsymbol{a}^{(n)}$ is also an eigenvector of $M$, with an eigenvalue or $-\|\boldsymbol{a}^{(n)}\|^2\cos^2\theta<0$.
Note that $\cos\theta,\sin\theta\ne0$, because we exclude $\theta=0, \pi/2, \pi$.
We can select $n-1$ vectors that are orthogonal to each other and that correspond to the latter eigenvalue.
Therefore, the quadratic form in question has rank $n\ge 5$ and is indefinite.
Meyer's theorem can be applied to this quadratic form, completing the proof of Theorem~\ref{thm1}(\ref{thm1.3-1}).

We proceed to the proof of Theorem~\ref{thm1}(\ref{thm1.3-2}).
The proof starts with the following proposition.

\begin{proposition}
Let $\boldsymbol{a}^{(3)}\ne\boldsymbol{0}^{(3)}$ be a three-dimensional integer vector and $\theta\in\Theta_3$ $(\theta\ne0, \pi/2, \pi)$.
The necessary and sufficient condition for $\theta\in\Theta_3(\boldsymbol{a}^{(3)})$ is 
that the ellipse in the $xy$-plane expressed as
\begin{equation}
\frac{\|\boldsymbol{a}^{(3)}\|^2x^2}{\|\boldsymbol{a}_\perp^{(3)}\|^2\tan^2\theta}+\frac{y^2}{\|\boldsymbol{a}_\perp^{(3)}\|^2\tan^2\theta}=1
\label{eq2.2}
\end{equation}
has a rational point $(x,y)\in\mathbb{Q}^2$.
Here, $\boldsymbol{a}_\perp^{(3)}$ is an arbitrary integer vector perpendicular to $\boldsymbol{a}^{(3)}$ and the existence of rational points on the ellipse does not depend on the choice of $\boldsymbol{a}_\perp^{(3)}$.
\label{prop2}
\end{proposition}

\begin{proof}
Suppose that integer vector $\boldsymbol{v}^{(3)}=(v_1, v_2, v_3)$ intersects with $\boldsymbol{a}^{(3)}$ at angle $\theta$.
By definition,
\[
(\boldsymbol{a}^{(3)}\cdot\boldsymbol{v}^{(3)})^2-\|\boldsymbol{a}^{(3)}\|^2\|\boldsymbol{v}^{(3)}\|^2\cos^2\theta=0.
\]
This equation is written as the quadratic form in $v_1, v_2$, and $v_3$:
\[
\begin{pmatrix}
v_1 & v_2 & v_3
\end{pmatrix}
\begin{pmatrix}
a_1^2-\|\boldsymbol{a}^{(3)}\|^2\cos^2\theta & a_1a_2 & a_1a_3\\
a_1a_2 & a_2^2-\|\boldsymbol{a}^{(3)}\|^2\cos^2\theta & a_2a_3\\
a_1a_3 & a_2a_3 & a_3^2-\|\boldsymbol{a}^{(3)}\|^2\cos^2\theta
\end{pmatrix}
\begin{pmatrix}
v_1 \\ v_2 \\ v_3
\end{pmatrix}
=0.
\]
According to a discussion similar to the above proof (1-3), the vector $\boldsymbol{a}^{(3)}$ is an eigenvector of the $3\times3$ matrix in the above equation with an eigenvalue of $\|\boldsymbol{a}^{(3)}\|^2\sin^2\theta$.
Each vector orthogonal to $\boldsymbol{a}^{(3)}$ is also an eigenvector and its eigenvalue is $-\|\boldsymbol{a}^{(3)}\|^2\cos^2\theta$.
Hence, if a certain $\boldsymbol{a}_\perp^{(3)}$ is fixed, $\{\boldsymbol{a}^{(3)}, \boldsymbol{a}_\perp^{(3)}, \boldsymbol{a}^{(3)}\times\boldsymbol{a}_\perp^{(3)}\}$ is an orthogonal basis that diagonalizes the matrix (here $\boldsymbol{a}^{(3)}\times\boldsymbol{a}_\perp^{(3)}$ is the vector product).
After diagonalization, we obtain
\begin{equation}
\begin{pmatrix}w_1 & w_2 & w_3 \end{pmatrix}
\begin{pmatrix}\sin^2\theta\\ & -\cos^2\theta\\ & & -\cos^2\theta\end{pmatrix}
\begin{pmatrix}w_1 \\ w_2 \\ w_3 \end{pmatrix}
=0,
\label{eq3}
\end{equation}
where
\[
w_1=\|\boldsymbol{a}_\perp^{(3)}\|\boldsymbol{a}^{(3)}\cdot\boldsymbol{v}^{(3)},\quad
w_2= \|\boldsymbol{a}^{(3)}\|\boldsymbol{a}_\perp^{(3)}\cdot\boldsymbol{v}^{(3)},\quad
w_3= (\boldsymbol{a}^{(3)}\times\boldsymbol{a}_\perp^{(3)})\cdot\boldsymbol{v}^{(3)}.
\]
Equation~\eqref{eq3} can be written as
\[
\frac{\|\boldsymbol{a}^{(3)}\|^2x^2}{\|\boldsymbol{a}_\perp^{(3)}\|^2\tan^2\theta}
+\frac{y^2}{\|\boldsymbol{a}_\perp^{(3)}\|^2\tan^2\theta}=1,
\]
where
\begin{equation}
x=\frac{\boldsymbol{a}_\perp^{(3)}\cdot\boldsymbol{v}^{(3)}}{\boldsymbol{a}^{(3)}\cdot\boldsymbol{v}^{(3)}},\quad
y=\frac{(\boldsymbol{a}^{(3)}\times\boldsymbol{a}_\perp^{(3)})\cdot\boldsymbol{v}^{(3)}}{\boldsymbol{a}^{(3)}\cdot\boldsymbol{v}^{(3)}}.
\label{eq2.4}
\end{equation}
When $\boldsymbol{a}^{(3)}, \boldsymbol{a}_\perp^{(3)}$, and $\boldsymbol{v}^{(3)}$ are integer vectors, the denominators and numerators of $x$ and $y$ are integers and $x$ and $y$ are rational.
Therefore, the existence of $\boldsymbol{v}^{(n)}$ corresponds to a rational point on ellipse~\eqref{eq2.2}.

Conversely, if ellipse~\eqref{eq2.2} has a rational point $(x, y)$, one can find integers $v_1, v_2$, and $v_3$ by solving Eq.~\eqref{eq2.4}.

Finally, we prove that whether ellipse~\eqref{eq2.2} has a rational point does not depend on the choice of $\boldsymbol{a}_\perp^{(3)}$.
Let a certain $\boldsymbol{a}_\perp^{(3)}$ be fixed, and $\boldsymbol{b}^{(3)}\ne\boldsymbol{a}_\perp^{(3)}$ be another integer vector perpendicular to $\boldsymbol{a}^(3)$.
We construct a bijective correspondence between the set of rational points on ellipse~\eqref{eq2.2} and that on the ellipse where $\boldsymbol{a}_\perp^{(3)}$ in Eq.~\eqref{eq2.2} is replaced by $\boldsymbol{b}^{(3)}$.
The vector $\boldsymbol{b}^{(3)}$ lies on the plane spanned by $\boldsymbol{a}_\perp^{(3)}$ and $\boldsymbol{a}^{(3)}\times\boldsymbol{a}_\perp^{(3)}$, and we can write
\[
\boldsymbol{b}^{(3)}=\alpha\boldsymbol{a}_\perp^{(3)}+\beta\boldsymbol{a}^{(3)}\times\boldsymbol{a}_\perp^{(3)},
\]
with rational numbers $\alpha$ and $\beta$ ($\alpha\beta\ne0$).
The squared norm of $\boldsymbol{b}^{(3)}$ is calculated to be
\[
\|\boldsymbol{b}^{(3)}\|^2=\alpha^2\|\boldsymbol{a}_\perp^{(3)}\|^2+\beta^2\|\boldsymbol{a}^{(3)}\times\boldsymbol{a}_\perp^{(3)}\|^2=(\alpha^2+\beta^2\|\boldsymbol{a}\|^2)\|\boldsymbol{a}_\perp^{(3)}\|^2,
\]
owing to the orthogonality of $\boldsymbol{a}^{(3)}, \boldsymbol{a}_\perp^{(3)}$, and $\boldsymbol{a}^{(3)}\times\boldsymbol{a}_\perp^{(3)}$.
Thus, the equation of the ellipse where $\boldsymbol{a}_\perp^{(3)}$ in Eq.~\eqref{eq2.2} is replaced by $\boldsymbol{b}^{(3)}$ is
\[
\frac{x'^2+\|\boldsymbol{a}^{(3)}\|^2y'^2}{(\alpha^2+\beta^2\|\boldsymbol{a}^{(3)}\|^2)\|\boldsymbol{a}_\perp^{(3)}\|^2\tan^2\theta}=1.
\]
If $(x,y)$ is a rational point on the ellipse for $\boldsymbol{a}_\perp^{(3)}$, then
\[
\begin{pmatrix}
x'\\
y'
\end{pmatrix}
=
\begin{pmatrix}
\alpha x+\beta y\\
\alpha y-\|\boldsymbol{a}^{(3)}\|^2\beta x
\end{pmatrix}
=\begin{pmatrix}
\alpha & \beta\\
-\|\boldsymbol{a}^{(3)}\|^2\beta & \alpha
\end{pmatrix}
\begin{pmatrix}
x\\
y
\end{pmatrix}
\]
is a rational point on the ellipse for $\boldsymbol{b}^{(3)}$.
This relation between $(x, y)$ and $(x', y')$ is bijective because the $2\times2$ matrix has determinant $\alpha^2+\|\boldsymbol{a}^{(3)}\|^2\beta^2>0$.
Thus, sets of rational points on the ellipse for different choices of $\boldsymbol{a}_\perp^{(3)}$ have a one-to-one correspondence with each other.
\end{proof}

Proposition~\ref{prop2} appears slightly unsatisfactory, in that the condition for $\theta\in\Theta_3(\boldsymbol{a}^{(3)})$ involves arbitrary $\boldsymbol{a}_\perp^{(3)}$, 
although we proved that the choice of $\boldsymbol{a}_\perp^{(3)}$ does not affect the result.
The statement of Proposition~\ref{prop2} is improved by introducing the Hilbert symbol, which determines the existence of rational points on a rational conic.
The basic properties of the Hilbert symbol can be found in Serre~\cite{Serre}.

\begin{lemma}
$\Theta_3(\boldsymbol{a}^{(3)})$ contains $\theta\in\Theta_3$ $(\theta\ne0,\pi/2,\pi)$ if and only if
\begin{equation}
(\|\boldsymbol{a}^{(3)}\|^2, \|\boldsymbol{a}^{(3)}\|^2)_p (\|\boldsymbol{a}^{(3)}\|^2, \tan^2\theta)_p (\tan^2\theta, \tan^2\theta)_p = 1
\label{eq2.5}
\end{equation}
holds for any odd prime $p$,
where $(\cdot,\cdot)_p$ denotes the Hilbert symbol.
\label{lemma1}
\end{lemma}

\begin{proof}
According to the property (or definition) of the Hilbert symbol, the ellipse in Eq.~\eqref{eq2.2} has a rational point if and only if
\begin{equation}
\left(\frac{\|\boldsymbol{a}_\perp^{(3)}\|^2\tan^2\theta}{\|\boldsymbol{a}^{(3)}\|^2}, \|\boldsymbol{a}_\perp^{(3)}\|^2\tan^2\theta\right)_v=1
\label{eq2.6}
\end{equation}
for every prime $v=p$ and $v=\infty$.
The case $v=\infty$ is true for any $\theta, \boldsymbol{a}^{(3)}$, and $\boldsymbol{a}_\perp^{(3)}$ because the two arguments of the Hilbert symbol are positive.
According to the product formula of the Hilbert symbol, it suffices to show the case in which $v$ is an odd prime $p$, and the $v=2$ case can be omitted.
By using properties of the Hilbert symbol, Eq.~\eqref{eq2.6} is reduced to
\[
(\|\boldsymbol{a}^{(3)}\|^2, \|\boldsymbol{a}_\perp^{(3)}\|^2)_p (\|\boldsymbol{a}_\perp^{(3)}\|^2, \|\boldsymbol{a}_\perp^{(3)}\|^2)_p
(\|\boldsymbol{a}^{(3)}\|^2,\tan^2\theta)_p (\tan^2\theta, \tan^2\theta)_p=1.
\]

Without loss of generality, we can assume $a_1\ne0$, and we set $\boldsymbol{a}_\perp^{(3)}=(a_2, -a_1,0)(\ne\boldsymbol{0}^{(3)})$ in the following.
(Recall that the existence of rational points on ellipse~\eqref{eq2.2} does not depend on the choice of $\boldsymbol{a}_\perp^{(3)}$.)
We have 
\[
(\|\boldsymbol{a}_\perp^{(3)}\|^2, \|\boldsymbol{a}_\perp^{(3)}\|^2)_p=(a_1^2+a_2^2, a_1^2+a_2^2)_p=1,
\]
because the circle $(a_1^2+a_2^2)(x^2+y^2)=1$ has a rational point $(x,y)=(a_1, a_2)/(a_1^2+a_2^2)$.
Next, we calculate $(\|\boldsymbol{a}^{(3)}\|^2, \|\boldsymbol{a}_\perp^{(3)}\|^2)_p$.
If $a_3=0$, we immediately get $(\|\boldsymbol{a}^{(3)}\|^2, \|\boldsymbol{a}_\perp^{(3)}\|^2)=(\|\boldsymbol{a}^{(3)}\|^2, \|\boldsymbol{a}^{(3)}\|^2)$ because $\|\boldsymbol{a}_\perp^{(3)}\|^2=\|\boldsymbol{a}^{(3)}\|^2$ in this case.
If $a_3\ne0$, the relation $\|\boldsymbol{a}_\perp^{(3)}\|^2=\|\boldsymbol{a}^{(3)}\|^2-a_3^2$ yields
\begin{align*}
(\|\boldsymbol{a}^{(3)}\|^2, \|\boldsymbol{a}_\perp^{(3)}\|^2)_p
&=\left(\|\boldsymbol{a}^{(3)}\|^2, 1-\frac{a_3^2}{\|\boldsymbol{a}^{(3)}\|^2}\right)_p (\|\boldsymbol{a}^{(3)}\|^2, \|\boldsymbol{a}^{(3)}\|^2)_p\\
&=(\|\boldsymbol{a}^{(3)}\|^2, \|\boldsymbol{a}^{(3)}\|^2)_p,
\end{align*}
where we applied the formula $(a, 1-a)_p=1$ in the last equality.
Thus, regardless of whether $a_3=0$ or not, we obtain
\[
(\|\boldsymbol{a}^{(3)}\|^2, \|\boldsymbol{a}_\perp^{(3)}\|^2)_p=(\|\boldsymbol{a}^{(3)}\|^2, \|\boldsymbol{a}^{(3)}\|^2)_p.
\]
This completes the proof.
\end{proof}

\begin{remark}
From this lemma, we find that the set $\Theta_3(\boldsymbol{a}^{(3)})$ depends only on the squared norm $\|\boldsymbol{a}^{(3)}\|^2$ and not on the individual components $a_1, a_2$, and $a_3$.
\end{remark}

Using Lemma~\ref{lemma1}, we provide the proof of Theorem~\ref{thm1}(\ref{thm1.3-2}).
For each integer vector $\boldsymbol{a}^{(3)}$, we prove that there exists $\theta\in\Theta_3$ ($\theta\ne0,\pi/2,\pi$) that does not satisfy Eq.~\eqref{eq2.5}.
We consider the following four cases separately according to the prime factorization of $\|\boldsymbol{a}^{(3)}\|^2$.
(Note that any $\|\boldsymbol{a}^{(3)}\|^2$ falls under at least one case and may match more than one case in \ref{case2}--\ref{case4}.)
\begin{enumerate}
\renewcommand{\labelenumi}{(\alph{enumi})}
\renewcommand{\theenumi}{\labelenumi}
\item When $\|\boldsymbol{a}^{(3)}\|^2$ is a square.

Because $(\|\boldsymbol{a}^{(3)}\|^2, \|\boldsymbol{a}^{(3)}\|^2)_p=(\|\boldsymbol{a}^{(3)}\|^2, \tan^2\theta)_p=1$ for any prime $p$ and $\theta$ in this case, $\Theta_3(\boldsymbol{a}^{(3)})$ does not contain $\theta$ whose squared tangent satisfies $(\tan^2\theta, \tan^2\theta)_p=-1$ for some $p$.
For example, $\theta=\pi/3$ ($\tan^2\theta=3$) belongs to $\Theta_3$, but not to $\Theta_3(\boldsymbol{a}^{(3)})$ because $(3,3)_3=-1$.
Since $\|(1,0,0)\|^2=1$, Proposition~\ref{prop1.1}, where $\Theta_3(1,0,0)\not\ni\pi/3$, is a special case of this result.

\item\label{case2} When a prime $p\equiv3\pmod4$ appears in the square-free part of $\|\boldsymbol{a}^{(3)}\|^2$ (equivalently, $p$ has an odd exponent in the factorization of $\|\boldsymbol{a}^{(3)}\|^2$).

In this case, $(\|\boldsymbol{a}^{(3)}\|^2, \|\boldsymbol{a}^{(3)}\|^2)_p=(p,p)_p=-1$.
If $\tan\theta$ is rational (namely, $\theta\in\Theta_2$), $(\|\boldsymbol{a}^{(3)}\|^2, \tan^2\theta)_p=(\tan^\theta, \tan^2\theta)_p=1$ always holds, so that Eq.~\eqref{eq2.5} is not satisfied.
Hence, $\Theta_3(\boldsymbol{a}^{(3)})$ in this case does not contain any $\theta\in\Theta_2$.

\item When a prime $p\equiv1\pmod4$ appears in the square-free part of $\|\boldsymbol{a}^{(3)}\|^2$.

In this case, $(\|\boldsymbol{a}^{(3)}\|^2, \|\boldsymbol{a}^{(3)}\|^2)_p=(p,p)_p=1$.
We take an integer $c$ such that $1\le c\le p-1$ and $(c,p)_p=\left(\dfrac{c}{p}\right)=-1$, where $\left(\dfrac{c}{p}\right)$ is the Legendre symbol~\cite{Serre}.
(Such $c$ always exists, because exactly half of integers in $1\le c\le p-1$ are $(p, c)_p=-1$.) 
We also have
\[
(p-c,p)_p=\left(\frac{p-c}{p}\right)=\left(\frac{-c}{p}\right)=(-1,p)_p(c,p)_p=-1,
\]
because $(-1,p)_p=1$ for $p\equiv1\pmod4$.
Since $c+(p-c)=p\equiv1\mod4$, at least either $c$ or $p-c$ is not of the form $4^k(8m+7)$.
Thus, by interchanging $c$ and $p-c$ if necessary, there exists an integer $c$ such that $(c,p)_p=-1$ and $c$ is not of the form $4^k(8m+7)$.
From this result, we can take the angle $\theta\in\Theta_3$ being $\tan^2\theta=c$, and have
\begin{align*}
&(\|\boldsymbol{a}^{(3)}\|^2, \|\boldsymbol{a}^{(3)}\|^2)_p(\|\boldsymbol{a}^{(3)}\|^2, \tan^2\theta_1)_p(\tan^2\theta_1, \tan^2\theta_1)_p\\
&= (p, p)_p (p, c)_p (c,c)_p = -1,
\end{align*}
which violates Eq.~\eqref{eq2.5}.
Hence, $\theta$ belongs to $\Theta_3$ but not to $\Theta_3(\boldsymbol{a}^{(3)})$.

\item\label{case4} When $2$ appears in the square-free part of $\|\boldsymbol{a}^{(3)}\|^2$.

We let $p$ be an odd prime which is coprime to $\|\boldsymbol{a}^{(3)}\|^2$; that is, $(\|\boldsymbol{a}^{(3)}\|^2, \|\boldsymbol{a}^{(3)}\|^2)_p=1$.
Taking the angle $\theta$ such that $\tan^2\theta=p$, we get
\begin{align*}
(\|\boldsymbol{a}^{(3)}\|^2, \tan^2\theta)_p (\tan^2\theta, \tan^2\theta)_p
&=(2, p)_p (p, p)_p\\
&= (-1)^{(p^2-1)/8} (-1)^{(p-1)/2}\\
&=\begin{cases}
+1 & p\equiv1,3\pmod8\\
-1 & p\equiv5,7\pmod8.
\end{cases}
\end{align*}
Hence, $\Theta_3(\boldsymbol{a}^{(3)})$ does not contain $\theta$ whose squared tangent is a prime number coprime to $\|\boldsymbol{a}^{(3)}\|^2$ and $\tan^2\theta\equiv5\pmod8$.
We note that the other choice of $\tan^2\theta\equiv7\pmod8$ is not an appropriate example of $\theta\in\Theta_3$ and $\theta\notin\Theta_3(\boldsymbol{a}^{(3)})$, because this $\theta$ does not belong to $\Theta_3$ (owing to the three-square theorem).
\end{enumerate}

\begin{remark}
Based on Lemma~\ref{lemma1}, if $\theta\in\Theta_3(\boldsymbol{a}^{(3)})\setminus\{0,\pi/2,\pi\}$, then the angle $\theta_n$ given by $\tan^2\theta_n=n^2\tan^2\theta$ $(n=1,2,3,\ldots)$ satisfies $\theta_n\in\Theta_3(\boldsymbol{a}^{(3)})$.
Similarly, $\theta\in\Theta_3\setminus\Theta_3(\boldsymbol{a}^{(3)})$ implies $\theta_n\in\Theta_3\setminus\Theta_3(\boldsymbol{a}^{(3)})$.
Hence, for any integer vector $\boldsymbol{a}^{(3)}\ne\boldsymbol{0}^{(3)}$, both sets $\Theta_3(\boldsymbol{a}^{(3)})$ and $\Theta_3\setminus\Theta_3(\boldsymbol{a}^{(3)})$ are infinite.
\end{remark}

If $\boldsymbol{b}^{(3)}$ is an integer vector satisfying $\angle(\boldsymbol{a}^{(3)}, \boldsymbol{b}^{(3)})=\theta$, then the integer multiplications $2\boldsymbol{b}^{(3)}, 3\boldsymbol{b}^{(3)},\ldots$ are also integer vectors that intersect with $\boldsymbol{a}^{(3)}$ at angle $\theta$.
Even if this trivial multiplicity is ignored, an infinite number of integer vectors intersect with $\boldsymbol{a}^{(3)}$ at angle $\theta$.
In other words, there exist an infinite number of integer vectors $\boldsymbol{v}^{(3)}$ with mutually different directions such that $\angle(\boldsymbol{a}^{(3)}, \boldsymbol{v}^{(3)})=\theta$.
In fact, from Eq.~\eqref{eq2.4}, $\boldsymbol{v}^{(3)}, 2\boldsymbol{v}^{(3)}, 3\boldsymbol{v}^{(3)},\ldots$ yield the same rational point $(x,y)$ of the ellipse~\eqref{eq2.2}, and different rational points correspond to integer vectors having different directions.
According to the theory of rational points on conics~\cite{Husemoller}, whenever one rational point is found on a rational conic, there exists infinitely and densely distributed rational points on it.
This result can be summarized as follows:
\begin{corollary}
If $\theta\in\Theta_3(\boldsymbol{a}^{(3)})$, then there exists an infinite number of integer vectors with mutually different directions that intersect with $\boldsymbol{a}^{(3)}$ at angle $\theta$.
\end{corollary}

Lemma~\ref{lemma1} or Proposition~\ref{prop2} describes the criterion for the existence of an integer vector $\boldsymbol{v}^{(3)}$ such that $\angle(\boldsymbol{a}^{(3)}, \boldsymbol{v}^{(3)})=\theta$.
When $\boldsymbol{a}^{(3)}$ and $\theta$ satisfy Eq.~\eqref{eq2.5}, this lemma cannot tell us how to find such a vector $\boldsymbol{v}^{(3)}$.
To obtain $\boldsymbol{v}^{(3)}$, we have to find a rational point $(x,y)$ on the ellipse~\eqref{eq2.2} and solve Eq.~\eqref{eq2.4}.
For finding rational points on a rational conic, see Kato, Kurokawa, Sait\=o, and Kurihara~\cite{Kato}.

\section{Lattice angles in dimension three}\label{sec3}
Recall that Theorem~\ref{thm1} states that the set $\Theta_n(\boldsymbol{a}^{(n)})$ of lattice angles is not equal to $\Theta_n$ only for $n=3$.
Determining $\Theta_3(\boldsymbol{a}^{(3)})$ is a next natural problem.
This section focuses on three-dimensional vectors, and we omit the superscript ``(3)'' and use simple symbols such as $\boldsymbol{a}$.
We explicitly write the condition for $\theta\in\Theta_3(\boldsymbol{a})$ when $\|\boldsymbol{a}\|^2$ has a simple prime factorization, as in Theorem~\ref{thm3} below.

\begin{definition}
The square-free part of a positive rational number $X$ is the minimum integer $N$ such that $X/N$ is a squared rational.

For example, the square-free part of $1/4$ is $1$ and that of $5/12$ is $15$.
\end{definition}

When the square-free part of $\tan^2\theta$ is $N$, the Hilbert symbols appearing in Eq.~\eqref{eq2.5} are calculated to be $(\|\boldsymbol{a}\|^2,\tan^2\theta)_p=(\|\boldsymbol{a}\|^2,N)_p$ and $(\tan^2\theta, \tan^2\theta)_p=(N, N)_p$.
Thus, whether $\theta\in\Theta_3$ is an element of $\Theta_3(\boldsymbol{a})$, that is, whether $\boldsymbol{a}$ has an integer vector that intersects at the angle $\theta\in\Theta_3$ depends only on the square-free part of $\tan^2\theta$.
In the following theorem, we assume that the square-free part of $\tan^2\theta$ is factorized as $p_1\cdots p_s$ or $2p_1\cdots p_s$, where $p_1,\ldots,p_s$ are odd primes.
These two forms can be expressed as $2^b p_1\cdots p_s$, where $b=0$ or $1$.
If $s=0$, we define as $p_1\cdots p_s=1$.

\begin{theorem}
We let $M$ be an integer and $b,b'\in\{0,1\}$.
\begin{enumerate}
\item When the squared norm of $\boldsymbol{a}$ is a perfect square \textup(i.e., $\|\boldsymbol{a}\|^2=M^2$\textup),
$\theta\in\Theta_3(\boldsymbol{a})\setminus\{0,\pi/2,\pi\}$ if and only if the square-free part of $\tan^2\theta$ is of the form $2^b p_1\cdots p_s$, with $p_i\equiv1\pmod4$.
\item When the squared norm of $\boldsymbol{a}$ is of the form $\|\boldsymbol{a}\|^2=2M^2$,
$\theta\in\Theta_3(\boldsymbol{a})\setminus\{0,\pi/2,\pi\}$ if and only if the square-free part of $\tan^2\theta$ is of the form $2^b p_1\cdots p_s$, with $p_i\equiv1, 3\pmod 8$.
\item When the squared norm of $\boldsymbol{a}$ is of the form $\|\boldsymbol{a}\|^2=qM^2$ where $q$ is an odd prime, the condition for $\theta\in\Theta_3(\boldsymbol{a})$ can be written according to the residue of $q$ divided by $8$.
\begin{enumerate}
\makeatletter
\renewcommand{\p@enumii}{}
\makeatother
\item $q\equiv3\pmod8$\textup: $\theta\in\Theta_3(\boldsymbol{a})\setminus\{0,\pi/2,\pi\}$ if and only if the square-free part of $\tan^2\theta$ is of the form $2q^{b'}p_1\cdots p_s$, with $\left(\dfrac{p_i}{q}\right)=1$ for $i=1,\ldots,s$.
%
\item\label{thm3.1-3-2} $q\equiv1\pmod8$\textup: $\theta\in\Theta_3(\boldsymbol{a})\setminus\{0,\pi/2,\pi\}$ if and only if the square-free part of $\tan^2\theta$ is of the form $2^b q^{b'} p_1\cdots p_s$, with $\left(\dfrac{p_i}{q}\right)=(-1)^{(p_i-1)/2}$ for $i=1,\ldots,s$ and $(p_1\cdots p_s-1)/2$ being even.
%
\item\label{thm3.1-3-3} $q\equiv5\pmod8$\textup:
$\theta\in\Theta_3(\boldsymbol{a})\setminus\{0,\pi/2,\pi\}$ if and only if
the square-free part of $\tan^2\theta$ is of the form $q^{b'} p_1\cdots p_s$ with $\left(\dfrac{p_i}{q}\right)=(-1)^{(p_i-1)/2}$ for $i=1,\ldots,s$ and $(p_1\cdots p_s-1)/2$ being even,
or $2q^{b'} p_1\cdots p_s$ with $\left(\dfrac{p_i}{q}\right)=(-1)^{(p_i-1)/2}$ for $i=1,\ldots,s$ and $(p_1\cdots p_s-1)/2$ being odd.
\end{enumerate}
Note that case $q\equiv7\pmod8$ does not occur because $\|\boldsymbol{a}\|^2=a_1^2+a_2^2+a_3^2$ is the sum of three squares.
\end{enumerate}
\label{thm3}
\end{theorem}

In statements (\ref{thm3.1-3-2}) and (\ref{thm3.1-3-3}), even or odd $(p_1\cdots p_s-1)/2$ indicates that there exists an even or odd number of primes $p_i\equiv3\pmod4$, respectively.

\begin{proof}

(1) Equation~\eqref{eq2.5} holds trivially for every prime other than $p_1,\ldots,p_s$; therefore, we have to check the case $p=p_i$.
When $\|\boldsymbol{a}\|^2=M^2$, the two Hilbert symbols in Eq.~\eqref{eq2.5} are $(\|\boldsymbol{a}\|^2, \|\boldsymbol{a}\|^2)_{p_i}=(\|\boldsymbol{a}\|^2, \tan^2\theta)_{p_i}=1$ for each $i=1,\ldots,s$.
Hence, $\theta\in\Theta_3(\boldsymbol{a})$ is equivalent to satisfying $(\tan^2\theta, \tan^2\theta)_{p_i}=1$ for $i=1,\ldots,s$.
This condition is further reduced to $(p_i, p_i)_{p_i}=1$, indicating that $p_i\equiv1\pmod4$.

(2) As with the above case (1), we check Eq.~\eqref{eq2.5} for $p=p_i$.
When $\|\boldsymbol{a}\|^2=2M^2$, $(\|\boldsymbol{a}\|^2, \|\boldsymbol{a}\|^2)_p=1$ for every odd prime $p$.
Hence,
\begin{align*}
(\|\boldsymbol{a}\|^2, \|\boldsymbol{a}\|^2)_{p_i}(\|\boldsymbol{a}\|^2, \tan^2\theta)_{p_i} (\tan^2\theta, \tan^2\theta)_{p_i}
&=(2, p_i)_{p_i} (p_i, p_i)_{p_i}\\
&=\begin{cases}
+1 & p_i\equiv1,3\pmod8\\
-1 & p_i\equiv5,7\pmod8.
\end{cases}
\end{align*}
This expression is found in case~\ref{case4} in the proof of Theorem~\ref{thm1}(\ref{thm1.3-2}).

(3) We have to write Eq.~\eqref{eq2.5} for $p=p_i$ and $p=q$ cases.

If the square-free part of $\tan^2\theta$ is coprime to $q$ ($b'=0$), then Eq.~\eqref{eq2.5} becomes
\begin{equation}
(q,q)_q (q, 2^b p_1\cdots p_s)_q=1\quad(\text{for $p=q$}), \quad
(q, p_i)_{p_i} (p_i, p_i)_{p_i} = 1 \quad(\text{for $p=p_i$}).
\label{eq6}
\end{equation}
Otherwise, if the square-free part of $\tan^2\theta$ includes $q$ ($\tan^2\theta=2^b q p_1\cdots p_s x^2$, where $x$ is a rational number),
Eq.~\eqref{eq2.5} becomes
\[
(q, 2^bq p_1\cdots p_s)_q=1\quad(\text{for $p=q$}), \quad
(q, p_i)_{p_i} (p_i, p_i)_{p_i} = 1 \quad(\text{for $p=p_i$}),
\]
which is the same as in Eq.~\eqref{eq6}.
Therefore, we calculate Eq.~\eqref{eq6}, regardless of whether the square-free part of $\tan^2\theta$ is coprime to $q$.
By using the properties of the Hilbert symbol, Eq.~\eqref{eq6} becomes
\begin{align}
&(-1)^{(q-1)/2}\left(\frac{2^b}{q}\right)\left(\frac{p_1}{q}\right)\cdots\left(\frac{p_s}{q}\right)=1,\label{eq7}\\
&(-1)^{(p_i-1)(q+1)/4}\left(\frac{p_i}{q}\right)=1,\label{eq8}
\end{align}
where the quadratic reciprocity law~\cite{Serre}
\[
\left(\frac{q}{p_i}\right)=\left(\frac{p_i}{q}\right)(-1)^{(p_i-1)(q-1)/4}
\]
is used to derive Eq.~\eqref{eq8}.

We further calculate Eqs.~\eqref{eq7} and \eqref{eq8} according to the residue of $q$ divided by $8$.
\begin{enumerate}
\renewcommand{\labelenumi}{(\alph{enumi})}
\item When $q\equiv3\pmod8$, Eq.~\eqref{eq8} becomes $\left(\displaystyle\frac{p_i}{q}\right)=1$, and Eq.~\eqref{eq7} is reduced to
\[
\left(\frac{2^b}{q}\right)=-1.
\]
$b=0$ does not satisfy this equation, whereas $b=1$ satisfies this equation for any $q\equiv3\pmod8$ because
\[
\left(\frac{2}{q}\right)=(-1)^{(q^2-1)/8}.
\]
\item When $q\equiv1\pmod8$, Eq.~\eqref{eq8} becomes
\[
\left(\frac{p_i}{q}\right)=(-1)^{(p_i-1)/2}.
\]
In this case, $(-1)^{(q-1)/2}=1$ and $\left(\dfrac{2^0}{q}\right)=\left(\dfrac{2^1}{q}\right)=1$.
Hence, Eq.~\eqref{eq7} is equivalent to
\[
(-1)^{(p_1-1)/2}\cdots(-1)^{(p_s-1)/2}=1,
\]
indicating that there are an even number of $i$ such that $p_i\equiv3\pmod4$.
\item When $q\equiv5\pmod8$, Eq.~\eqref{eq8} becomes
\[
\left(\frac{p_i}{q}\right)=(-1)^{(p_i-1)/2},
\]
which is the same as the above case.
However, in this case
$\left(\dfrac{2^0}{q}\right)=1$ and $\left(\dfrac{2^1}{q}\right)=-1$.
By considering this difference, the proof can be conducted in an analogous manner to the above case.
\end{enumerate}
\end{proof}

$\Theta_3(\boldsymbol{a})$ for a more general form of $\|\boldsymbol{a}\|^2$, even for $\|\boldsymbol{a}\|^2=q_1q_2M^2$ with distinct primes $q_1$ and $q_2$, is expected to become more complicated because $q_1$ and $q_2$ affect each other.
The general and adequate characterization of $\Theta_3(\boldsymbol{a})$ is an open problem.
In relation to this problem, the following property is easy to prove, although this is far from the complete determination of $\Theta_3(\boldsymbol{a})$ for general $\|\boldsymbol{a}\|^2$.

\begin{proposition}
Let $\boldsymbol{a}_1, \boldsymbol{a}_2, \boldsymbol{a}_3$, and $\boldsymbol{b}$ be integer vectors that satisfy $\|\boldsymbol{b}\|^2=\|\boldsymbol{a}_1\|^2\|\boldsymbol{a}_2\|^2\|\boldsymbol{a}_3\|^2$.
Any angle $\theta\in\Theta_3(\boldsymbol{a}_1)\cap\Theta_3(\boldsymbol{a}_2)\cap\Theta_3(\boldsymbol{a}_3)$ belongs to $\Theta_3(\boldsymbol{b})$;
that is, $\Theta_3(\boldsymbol{a}_1)\cap\Theta_3(\boldsymbol{a}_2)\cap\Theta_3(\boldsymbol{a}_3)\subset\Theta_3(\boldsymbol{b})$.
\end{proposition}
\begin{proof}
Suppose $\theta\in\bigcap_{i=1}^3\Theta_3(\boldsymbol{a}_i)$ and $\theta\ne0,\pi/2,\pi$.
Using properties of the Hilbert symbol, especially the identity $(\tan^2\theta, \tan^2\theta)_p^2=1$, one obtains
\begin{align*}
(\|\boldsymbol{b}\|^2, \|\boldsymbol{b}\|^2)_p(\|\boldsymbol{b}\|^2, \tan^2\theta)_p(\tan^2\theta, \tan^2\theta)_p\\
&=\prod_{i=1}^3 (\|\boldsymbol{a}_i\|^2, \|\boldsymbol{a}_i\|^2)_p(\|\boldsymbol{a}_i\|^2, \tan^2\theta)_p(\tan^2\theta, \tan^2\theta)_p\\
&=1
\end{align*}
for any odd prime $p$.
\end{proof}

\begin{remark}
This proposition provides only a sufficient condition for $\theta\in\Theta_3(\boldsymbol{b})$.
Although $\theta\not\in\bigcap_{i=1}^3\Theta_3(\boldsymbol{a}_i)$, it is possible that $\theta$ belongs to $\Theta_3(\boldsymbol{b})$.
\end{remark}

Thus far, we have studied $\Theta_3(\boldsymbol{a})$ for a given integer vector $\boldsymbol{a}$.
Conversely, we can also investigate the set of integer vectors that can form a given angle $\theta$ with another integer vector:
$\{\boldsymbol{v}\in\mathbb{Z}^3\mid \angle(\boldsymbol{v}, \boldsymbol{w})=\theta \text{ for some $\boldsymbol{w}\in\mathbb{Z}^3$}\}=\{\boldsymbol{v}\in\mathbb{Z}^3\mid \Theta_3(\boldsymbol{v})\ni\theta\}$.
Since Eq.~\eqref{eq2.5} suggests that the squared norm $\|\boldsymbol{a}\|^2$ is more relevant to $\theta$ than the vector $\boldsymbol{a}$ itself, we introduce
\[
S(\theta):=\{\|\boldsymbol{v}\|^2\mid \boldsymbol{v}\in\mathbb{Z}^3\setminus\{\boldsymbol{0}\}, \Theta_3(\boldsymbol{v})\ni\theta\}.
\]
Recall that $0, \pi/2, \pi\in\Theta_3(\boldsymbol{v})$ for any integer vector $\boldsymbol{v}\ne\boldsymbol{0}$ (see the beginning of \S\ref{sec2}).
Thus, we immediately have
\[
S(0)=S\left(\frac{\pi}{2}\right)=S(\pi)=\{\|\boldsymbol{v}\|^2\mid \boldsymbol{v}\in\mathbb{Z}^3\setminus\{\boldsymbol{0}\}\}
=\{v_1^2+v_2^2+v_3^2\mid (v_1, v_2, v_3)\in\mathbb{Z}^3\}\setminus\{0\}.
\]
According to the three-square theorem, this is the set of positive integers that are not of the form $4^k(8m+7)$ ($k$ and $m$ are non-negative integers).
Moreover, we introduce the following four sets.
\begin{align*}
&\tilde\Theta_3:=\{\angle(\boldsymbol{a}, \boldsymbol{b})\mid \boldsymbol{a},\boldsymbol{b}\in\mathbb{Q}^3\}\setminus\left\{0,\frac{\pi}{2},\pi\right\},\\
&\tilde\Theta_3(\boldsymbol{a}):=\{\angle(\boldsymbol{a}, \boldsymbol{b})\mid \boldsymbol{b}\in\mathbb{Q}^3\}\setminus\left\{0,\frac{\pi}{2},\pi\right\}, \\
&\tilde{S}:=\{\|\boldsymbol{v}\|^2\mid \boldsymbol{v}\in\mathbb{Q}^3\setminus\{\boldsymbol{0}\}\},\\
&\tilde{S}(\theta):=\{\|\boldsymbol{v}\|^2\mid \boldsymbol{v}\in\mathbb{Q}^3, \tilde\Theta_3(\boldsymbol{v})\ni\theta\}
\end{align*}
Here, the tilde symbol indicates that the vectors involved have rational components.

\begin{proposition}
Let $\boldsymbol{a}\ne\boldsymbol{0}$ be an integer vector and $\theta\in\Theta_3\setminus\{0,\pi/2,\pi\}$.
\begin{enumerate}
\item\label{prop3.4-1} $\tilde\Theta_3=\Theta_3\setminus\{0,\pi/2,\pi\}$ and $\tilde\Theta_3(\boldsymbol{a})=\Theta_3(\boldsymbol{a})\setminus\{0,\pi/2,\pi\}$.
That is, the difference between $\Theta_3$ and $\tilde\Theta_3$ and between $\Theta_3(\boldsymbol{a})$ and $\tilde\Theta_3(\boldsymbol{a})$ is whether $0, \pi/2$, and $\pi$ are contained.
\item\label{prop3.4-2} $\tilde{S}=\tan^2\tilde\Theta_3:=\{\tan^2\theta\mid \theta\in\tilde\Theta_3\}$.
\item $\tilde{S}(\theta)\subseteq\tilde{S}$.
\item $S(\theta)=\tilde{S}(\theta)\cap\mathbb{Z}$.
\end{enumerate}
\label{prop3.4}
\end{proposition}

\begin{proof}
\begin{enumerate}
\item $\tilde\Theta_3\supseteq\Theta_3\setminus\{0, \pi/2, \pi\}$ holds trivially because integer vectors can be regarded as rational vectors.
Conversely, the angle between two vectors remains unchanged if these vectors are multiplied by positive integers.
Any rational vector becomes an integer vector by suitably multiplying an integer.
Hence, $\tilde\Theta_3\subseteq\Theta_3\setminus\{0, \pi/2, \pi\}$ holds.

A similar argument applies to $\Theta_3(\boldsymbol{a})$ and $\tilde\Theta_3(\boldsymbol{a})$.

\item $\tilde{S}$ is the set of positive rational numbers represented in the sum of three squared rationals.
Meanwhile, from Beeson's theorem~\cite{Beeson}, $\tilde\Theta_3$ is the set of angles whose squared tangent is represented in the sum of three squared rationals.
Hence, $\tilde{S}=\tan^2\tilde\Theta_3$.
\item Trivial by the definitions of $\tilde{S}$ and $\tilde{S}(\theta)$.
\item Since an integer vector is a kind of rational vector, $S(\theta)\subseteq\tilde{S}(\theta)$.
Taking the intersection with $\mathbb{Z}$, we have $S(\theta)\subseteq\tilde{S}(\theta)\cap\mathbb{Z}$, because $S(\theta)\cap\mathbb{Z}=S(\theta)$.

Conversely, we take an arbitrary element $N\in\tilde{S}(\theta)\cap\mathbb{Z}$.
By definition, there exists a rational vector $\boldsymbol{v}=(v_1, v_2, v_3)$ such that $v_1^2+v_2^2+v_3^2=N$.
We let $d$ be the lowest common denominator of $v_1, v_2$, and $v_3$; we write $v_i=w_i/d$ with $w_i\in\mathbb{Z}$ for $i=1,2,3$.
Immediately, $w_1^2+w_2^2+w_3^2=Nd^2$, and $N$ is found to be not of the form $4^k(8m+7)$, owing to the three-square theorem.
Thus, there exists an integer vector $\boldsymbol{a}=(a_1,a_2,a_3)$ such that $\|\boldsymbol{a}\|^2=a_1^2+a_2^2+a_3^2=N\in S(\theta)$.
\end{enumerate}
\end{proof}

\begin{lemma}[Duality of $\tilde{S}(\theta)$ and $\tan^2\tilde\Theta_3(\boldsymbol{a})$]
When the square-free part of $\|\boldsymbol{a}\|^2$ is equal to that of $\tan^2\theta$, we have that
\[
\tilde{S}(\theta)=\tan^2\tilde\Theta_3(\boldsymbol{a}):=\{\tan^2\theta\mid \theta\in\tilde\Theta_3(\boldsymbol{a})\}.
\]
\label{lemma2}
\end{lemma}

\begin{proof}
From the condition that the square-free part of $\|\boldsymbol{a}\|^2$ is equal to that of $\tan^2\theta$,
\begin{equation}
(\|\boldsymbol{a}\|^2, A)_p = (\tan^2\theta, A)_p
\label{eq9}
\end{equation}
for every prime $p$ and any non-zero rational number $A$.
For any rational vector $\boldsymbol{v}\in\mathbb{Q}^3$ such that $\|\boldsymbol{v}\|^2\in\tilde{S}(\theta)$, we have $\|\boldsymbol{v}\|^2\in \tilde{S}(\theta)\subseteq\tilde{S}=\tan^2\tilde\Theta_3$ by Proposition~\ref{prop3.4}.
Hence, there exists $\alpha\in\tilde\Theta_3$ such that $\|\boldsymbol{v}\|^2=\tan^2\alpha$.

From Lemma~\ref{lemma1}, $\|\boldsymbol{v}\|^2\in\tilde{S}(\theta)$ is equivalent to
\[
(\|\boldsymbol{v}\|^2, \|\boldsymbol{v}\|^2)_p(\|\boldsymbol{v}\|^2,\tan^2\theta)_p(\tan^2\theta, \tan^2\theta)_p=1
\]
for any odd prime $p$.
Using $\|\boldsymbol{v}\|^2=\tan^2\alpha$ and Eq.~\eqref{eq9}, we obtain
\[
(\|\boldsymbol{a}\|^2, \|\boldsymbol{a}\|^2)_p(\|\boldsymbol{a}\|^2, \tan^2\alpha)_p(\tan^2\alpha, \tan^2\alpha)_p=1.
\]
Therefore, $\|\boldsymbol{v}\|^2=\tan^2\alpha\in\tan^2\tilde\Theta_3(\boldsymbol{a})$, and we conclude that $\tilde{S}(\theta)\subseteq\tan^2\tilde\Theta_3(\boldsymbol{a})$.

We can prove the converse, $\tilde{S}(\theta)\supseteq\tan^2\tilde\Theta_3(\boldsymbol{a})$, in a similar manner as above.
\end{proof}

As a dual statement of Theorem~\ref{thm3}, we obtain the condition for integer vector $\boldsymbol{a}$ to form a given angle $\theta$ with another integer vector.
\begin{theorem}
Let $q_1,\ldots,q_t$ be odd primes and $b,b'\in\{0,1\}$.
\begin{enumerate}
\item\label{thm3.6-1} When $\tan^2\theta$ is a squared rational, $\Theta_3(\boldsymbol{a})\ni\theta$ if and only if the square-free part of $\|\boldsymbol{a}\|^2$ is of the form $2^bq_1\cdots q_t$, with $q_j\equiv1\pmod4$.

\item When the square-free part of $\tan^2\theta$ is $2$ \textup(i.e., $\tan^2\theta=2x^2$ for some non-zero rational number $x$\textup), $\Theta_3(\boldsymbol{a})\ni\theta$ if and only if the square-free part of $\|\boldsymbol{a}\|^2$ is of the form $2^b q_1\cdots q_t$, with $q_j\equiv1,3\pmod8$.

\item\label{thm3.6-3} When the square-free part of $\tan^2\theta$ is odd prime $p$ \textup(i.e., $\tan^2\theta=px^2$ for some non-zero rational number $x$\textup), the condition for $\Theta_3(\boldsymbol{a})\ni\theta$ is written according to the residue of $p$ by 8.
\begin{enumerate}
\item $p\equiv3\pmod8$\textup: $\Theta_3(\boldsymbol{a})\ni\theta$ if and only if the square-free part of $\|\boldsymbol{a}\|^2$ is of the form $2p^bq_1\cdots q_t$, with $\left(\dfrac{q_j}{p}\right)=1$ for $j=1,\ldots,t$.
\item $p\equiv1\pmod8$\textup: $\Theta_3(\boldsymbol{a})\ni\theta$ if and only if the square-free part of $\|\boldsymbol{a}\|^2$ is of the form $2^bp^{b'}q_1\cdots q_t$, with $\left(\dfrac{q_j}{p}\right)=(-1)^{(q_j-1)/2}$ for $j=1,\ldots,t$ and $(q_1\cdots q_t-1)/2$ being even.
\item $p\equiv5\pmod8$\textup: $\Theta_3(\boldsymbol{a})\ni\theta$ if and only if the square-free part of $\|\boldsymbol{a}\|^2$ is of the form $p^b q_1\cdots q_t$, with $\left(\dfrac{q_j}{p}\right)=(-1)^{(q_j-1)/2}$ for $j=1,\ldots,t$ and $(q_1\cdots q_t-1)/2$ being even, or $2p^b q_1\cdots q_t$, with $\left(\dfrac{q_j}{p}\right)=(-1)^{(q_j-1)/2}$ for $j=1,\ldots,t$ and $(q_1\cdots q_t)/2$ being odd.
\end{enumerate}
\end{enumerate}
\label{thm4}
\end{theorem}

\begin{proof}
For a given $\theta\in\Theta_3$, let $\boldsymbol{v}$ be an integer vector whose square-free part is equal to that of $\tan^2\theta$.
Using Lemma~\ref{lemma2} and Proposition~\ref{prop3.4}, we obtain
\[
S(\theta)=\tilde{S}(\theta)\cap\mathbb{Z}=\tan^2\tilde\Theta_3(\boldsymbol{v})\cap\mathbb{Z}.
\]
Hence, as a result of duality, finding an integer vector $\boldsymbol{a}$ such that $\|\boldsymbol{a}\|^2\in S(\theta)$, namely $\Theta_3(\boldsymbol{a})\ni\theta$, is equivalent to finding the angle $\alpha\in\Theta_3(\boldsymbol{v})$ whose tangent squared $\tan^2\alpha$ is a positive integer.
Each case in the statement corresponds to that in Theorem~\ref{thm3}.
\end{proof}

\begin{remark}
The angles $\pi/6$ ($=30^\circ$), $\pi/4$ ($=45^\circ$), and $\pi/3$ ($=60^\circ$) are fundamentally important to trigonometry.
As a direct application of Theorem~\ref{thm4}, we can explicitly write the conditions for an integer vector $\boldsymbol{a}$ to have an integer vector intersecting at these angles.
\begin{enumerate}
\item Since $\tan^2 45^\circ=1$ is a perfect square, Theorem~\ref{thm4}(\ref{thm3.6-1}) is directly applied.
Thus, $\Theta_3(\boldsymbol{a})\ni45^\circ$ if and only if the square-free part of $\|\boldsymbol{a}\|^2$ is of the form $q_1\cdots q_t$ or $2q_1\cdots q_t$ with $q_j\equiv1\pmod4$ for $j=1,\ldots,t$.
\item For $\theta=30^\circ$ and $60^\circ$, the square-free part of $\tan^2\theta$ is $3$ for both angles ($\tan^2 60^\circ=3$ and $\tan^2 30^\circ=1/3$), which falls into the case $p\equiv3\pmod8$ of Theorem~\ref{thm4}(\ref{thm3.6-3}).
The Legendre symbol in this case becomes
\[
\left(\frac{q}{3}\right)
=\begin{cases} 1 & q\equiv1\pmod3,\\ -1 & q\equiv2\pmod3.\end{cases}
\]
Thus, $\Theta_3(\boldsymbol{a})\ni 30^\circ, 60^\circ$ if and only if its square-free part $\|\boldsymbol{a}\|^2$ is of the form $2q_1\cdots q_t$ or $2\cdot3q_1\cdots q_t$ with $q_j\equiv1\pmod3$ for $j=1,\ldots,t$.
\end{enumerate}
Thus, whether $\Theta_3(\boldsymbol{a})$ contains $\pi/6, \pi/4$, or $\pi/3$ can be determined simply and easily by the residues of the square-free factors of $\|\boldsymbol{a}\|^2$.
\end{remark}

Another application of duality (Lemma~\ref{lemma2}) is the intersection of $\Theta_3(\boldsymbol{a})$ with all integer vectors $\boldsymbol{a}\ne\boldsymbol{0}$.
As shown in the beginning of the proof of Theorem~\ref{thm1} in \S\ref{sec2}, each integer vector $\boldsymbol{a}\ne\boldsymbol{0}$ satisfies $\Theta_3(\boldsymbol{a})\ni 0, \pi/2, \pi$. 
Meanwhile, the following proposition states that no angles other than $0, \pi/2$, and $\pi$ belong to $\Theta_3(\boldsymbol{a})$ for all $\boldsymbol{a}\ne\boldsymbol{0}$.
\begin{proposition}
\[
\bigcap_{\substack{\boldsymbol{a}\in\mathbb{Z}^3\\ \boldsymbol{a}\ne\boldsymbol{0}^{(3)}}} \Theta_3(\boldsymbol{a}) = \left\{0,\frac{\pi}{2},\pi\right\}
\]
\label{prop3.8}
\end{proposition}
\begin{proof}
From the result $\tilde\Theta_3(\boldsymbol{a})=\Theta_3(\boldsymbol{a})\setminus\{0,\pi/2,\pi\}$ in Proposition~\ref{prop3.4}(\ref{prop3.4-1}), it suffices to show that
\[
\bigcap_{\substack{\boldsymbol{a}\in\mathbb{Z}^3\\ \boldsymbol{a}\ne\boldsymbol{0}^{(3)}}} \tilde\Theta_3(\boldsymbol{a})=\emptyset.
\]
For any rational vector $\boldsymbol{v}$, there exists an integer vector $\boldsymbol{a}$ that has the same direction as $\boldsymbol{v}$.
Hence,
\[
\bigcap_{\substack{\boldsymbol{a}\in\mathbb{Z}^3\\ \boldsymbol{a}\ne\boldsymbol{0}^{(3)}}} \tan^2\tilde\Theta_3(\boldsymbol{a})
=\bigcap_{\substack{\boldsymbol{v}\in\mathbb{Q}^3\\ \boldsymbol{v}\ne\boldsymbol{0}^{(3)}}} \tan^2\tilde\Theta_3(\boldsymbol{v})
=\bigcap_{\|\boldsymbol{v}\|^2\in\tilde{S}} \tan^2\tilde\Theta_3(\boldsymbol{v})
=\bigcap_{\theta\in\tilde\Theta_3} \tilde{S}(\theta),
\]
where Proposition~\ref{prop3.4}(\ref{prop3.4-2}) and Lemma~\ref{lemma2} are used in the last equality.

Suppose that some rational vector $\boldsymbol{v}$ satisfies
\[
\|\boldsymbol{v}\|^2\in\bigcap_{\theta\in\tilde\Theta_3} \tilde{S}(\theta).
\]
This implies that any angle $\theta\in\tilde\Theta_3$ belongs to $\tilde\Theta_3(\boldsymbol{v})$; that is, $\tilde\Theta_3\subseteq\tilde\Theta_3(\boldsymbol{v})$.
This result contradicts Theorem~\ref{thm1}(\ref{thm1.3-2}).
Therefore,
\[
\bigcap_{\substack{\boldsymbol{a}\in\mathbb{Z}^3\\ \boldsymbol{a}\ne\boldsymbol{0}^{(3)}}} \tan^2\tilde\Theta_3(\boldsymbol{a})
=\bigcap_{\theta\in\tilde\Theta_3} \tilde{S}(\theta)=\emptyset.
\]
\end{proof}

From Theorem~\ref{thm1}(\ref{thm1.3-1}), we have
\[
\bigcap_{\substack{\boldsymbol{a}^{(n)}\in\mathbb{Z}^n\\ \boldsymbol{a}^{(n)}\ne\boldsymbol{0}^{(n)}}} \Theta_n(\boldsymbol{a}^{(n)}) = \Theta_n
\]
for $n\ne3$, which is in clear contrast to Proposition~\ref{prop3.8}.
This result also indicates the exceptionality of the three-dimensional lattice.

\section*{Acknowledgments}
The author is grateful to Dr.~Yoshinori Mishiba for the informative discussions and to a referee who suggested a plainer proof of the $n=4$ case of Theorem~\ref{thm1}(\ref{thm1.3-1}).

\end{document}